\documentclass[12pt, a4paper]{article}

\usepackage[intlimits]{amsmath}
\usepackage{amsthm,amsfonts,color}

\usepackage{url,hyperref}

\usepackage{graphicx}
\newcommand{\myimage}[2]{\includegraphics[height=#1]{#2}} % avec dessins
\DeclareMathOperator{\linspan}{Span}

\renewcommand{\Tilde}{\widetilde}

\renewcommand{\Bar}{\overline}

\newcommand{\RR}{\mathbb{R}}

\newcommand{\NN}{\mathbb{N}}
\newcommand{\cH}{\mathcal{H}}

\newtheorem{theorem}{Theorem}[section]
\newtheorem{prop}[theorem]{Proposition}
\newtheorem{lemma}[theorem]{Lemma}

\theoremstyle{definition}

\newtheorem{remark}[theorem]{Remark}

\DeclareMathOperator{\spec}{spec}
\DeclareMathOperator{\dist}{dist}
\DeclareMathOperator{\supp}{supp}

\begin{document}

\title{\bf Tunneling between corners\\ for Robin Laplacians}

\author{\Large\sc Bernard Helffer  and \Large\sc Konstantin Pankrashkin\\ [\medskipamount]
\it Laboratoire de math\'ematiques (UMR 8628), Universit\'e Paris-Sud, \\
\it B\^atiment 425, 91405 Orsay Cedex, France.\\[\medskipamount]
E-mail: \texttt{bernard.helffer@math.u-psud.fr},\\
\texttt{konstantin.pankrashkin@math.u-psud.fr}
}

\date{}

\maketitle

\begin{abstract}
We study the  Robin Laplacian in a domain with two corners of the same opening,
and we calculate the asymptotics of the two lowest eigenvalues
as the distance between the corners increases to infinity.
\end{abstract}

\section{Introduction}

Let $\Omega\subset\RR^d$ be an open set with a sufficiently regular boundary (e.g. compact Lipschitz 
or non-compact with a suitable  behavior at infinity) and $\beta\in\RR$.
By the associated Robin Laplacian $H_\beta\equiv H(\Omega,\beta)$ we mean the operator
acting  in a weak sense as
\[
H_\beta f:=-\Delta f, \quad \dfrac{\partial f}{\partial n}=\beta f \text{ at } \partial\Omega\,,
\]
where $n$ is the unit outward normal at the boundary; a rigorous definition
is given below (Subsection~\ref{poly}).
In various applications, such as the study of the critical temperature
in the enhanced surface superconductivity (and in this context the Robin condition is also called the De Gennes condition, see \cite{Ka} and references therein) or the analysis of certain reaction-diffusion
processes, one is interested in the spectral properties of $H_\beta$,
the behavior of the spectrum as $\beta\to+\infty$ being of a particular importance~\cite{gs1,lacey}.
For sufficiently regular $\Omega\,,$ it was shown in~\cite{lp} that
the bottom of the spectrum $E(\beta)$ behaves as
\[
E(\beta)=-C_\Omega \beta^2+o(\beta^2) \text{ as } \beta\to+\infty\,,
\]
where $C_\Omega>0$ is a constant depending on the geometry of the boundary.
In particular, $C_\Omega=1$ for smooth domains, and some information
on the subsequent terms of the asymptotics was obtained e.g. in~\cite{exner,FK,kp}.
In the non-smooth case one can have $C_\Omega>1$, and the
constant is understood better in the $2D$ case. If $\omega$ denotes
the minimal corner at the boundary, then
\[
C_\Omega=\dfrac{2}{1-\cos\omega} \text{ if } \omega<\pi\,,
\text{ and } C_\Omega=1 \text{ otherwise}.
\]
In other words, intuitively, each corner at the boundary can be viewed as a geometric well, and 
it is the deepest well which determines the principal term of the spectral asymptotics,
and one may expect that the respective vertices serve as the asymptotic support of the respective eigenfunction.
One meets the natural question of what happens if one has several wells of the same depth, i.e.
several corners with  the same opening. 
Similar questions appear in various settings: semiclassical limit for multiple
wells~\cite{hs1,hs2,bhbook,ank,bds},
distant potential perturbations~\cite{daumer}, domains coupled by a thin tube~\cite{bhm}
or waveguides with distant boundary perturbations~\cite{BE},
in which the interaction between wells gives rise to an exponentially small difference between
the lowest eigenvalues. The aim of the present paper is to obtain a result in the same spirit
for Robin Laplacians in a class of corner domains.
We note that the eigenvalues $E(\Omega,\beta)$ of $H(\Omega,\beta)$ satisfy the obvious scaling relation,
\begin{equation}
      \label{eq-scale}
E(\Omega,\ell \beta)=\ell^2 E(\ell\Omega,\beta)\,, \quad \ell>0\,,
\end{equation}
and the regime $\beta\to+\infty$ is essentially equivalent to the study
of $E(\ell\Omega,\beta)$ as $\ell\to+\infty\,$ with a fixed $\beta$. We prefer to deal with scaled domains
in order to have finite limits.

\begin{figure}[t]
\centering
\myimage{40mm}{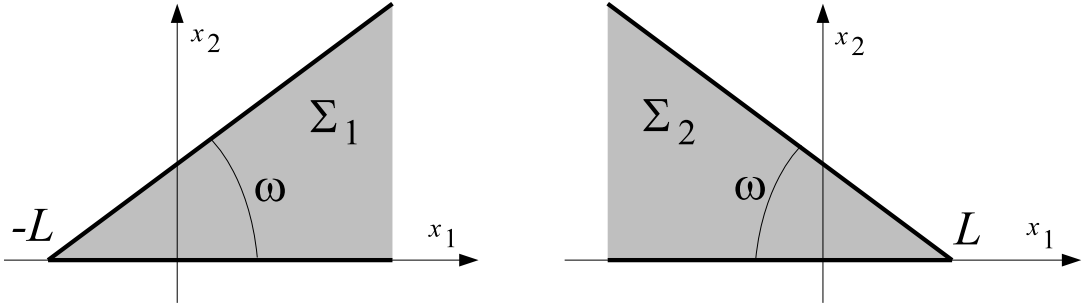}
\caption{The infinite sectors $\Sigma_1$ and $\Sigma_2$.}\label{fig1}
\end{figure}

Let us describe our result.
Let $\omega\in (0,\pi)$ and $L>0$. Denote by $\Omega_L$ the intersection of
the two infinite sectors $\Sigma_1$ and $\Sigma_2$,
\begin{align*}
\Sigma_1&:=\Big\{ (x_1,x_2):
\arg\big((x_1+ L) +i x_2\big)\in (0,\omega)
\Big\},\\
\Sigma_2&:=\big\{ (x_1,x_2):
(-x_1,x_2)\in \Sigma_1\big\},
\end{align*}
see Fig.~\ref{fig1}. Clearly, for $\omega\ge \pi/2$ the set $\Omega_L$ is
an infinite biangle whose vertices are the points
$A_1=(-L,0)$ and  $A_2=(L,0)$,
while
for $\omega<\pi/2$ we obtain the interior of the triangle
whose vertices are the above points $A_{1}$ and $A_2$
and the point $A_3=(0,L \tan\omega)$,
see Figure~\ref{fig2}. Let us fix some $\beta>0$.
The associated Robin Laplacian
\[
H_L:=H(\Omega_L,\beta)
\] is a self-adjoint operator
in $L^2(\Omega_L;\mathbb R)$, see Subsection~\ref{poly} for the rigorous definition.
Elementary considerations show that if $\omega<\pi/2$, then $H_L$ has a compact resolvent,
and the spectrum consists of eigenvalues $E_1(L)<E_2(L)\le \dots$. As usually, each eigenvalue
may appear several times according to its multiplicity.
For $\omega\ge\pi/2$ one has 
\[
\spec_\text{ess}H_L=[-\beta^2,+\infty)\,,
\]
 so the
discrete spectrum consists of  eigenvalues $E_1(L)<E_2(L)\le \dots<-\beta^2\,$.

Our main  result is as  follows:
\begin{theorem}\label{thm1}
Assume that either $\omega\in\big(0,\frac\pi 3\big)$ or $\omega\in\big[\frac\pi 2,\pi\big)$. Then, the two lowest eigenvalues satisfy, 
as $L\to+\infty$, 
\begin{align*}
E_1(L)&=-\dfrac{2\beta^2}{1-\cos\omega}\\
&\quad-4\beta^2\dfrac{1+\cos\omega}{(1-\cos\omega)^2}
\exp\Big(-2\beta\dfrac{1+\cos\omega}{\sin\omega} L\Big) +
O\bigg(L^{2} \exp\Big(-(2+\delta)\beta\dfrac{1+\cos\omega}{\sin\omega} L\Big)\bigg),\\
E_2(L)&=-\dfrac{2\beta^2}{1-\cos\omega}\\
&\quad +4\beta^2\dfrac{1+\cos\omega}{(1-\cos \omega)^2}
\exp\Big(-2\beta\dfrac{1+\cos\omega}{\sin\omega} L\Big) +
O\bigg(L^{2} \exp\Big(-(2+\delta)\beta\dfrac{1+\cos\omega}{\sin\omega} L\Big)\bigg),\\
\end{align*}
where $\delta=2\big((\cos \omega)^{-1}-1\big)$ for $\theta<\pi/3$ and $\delta=2$
for $\omega\ge \pi/2\,$.
 In particular,
\begin{multline*}
E_2(L)-E_1(L)=
8\beta^2\dfrac{1+\cos\omega}{(1-\cos\omega)^2}
\exp\Big(-2\beta\dfrac{1+\cos\omega}{\sin\omega} L\Big)\\ +
O\bigg(L^{2} \exp\Big(-(2+\delta)\beta\dfrac{1+\cos\omega}{\sin\omega} L\Big)\bigg).
\end{multline*}
\end{theorem}

Our proof is in the spirit of the scheme developed by Helffer and Sj\"ostrand for the semiclassical analysis
of the multiple well problem~\cite{hs1,bhbook}. In Section~\ref{prel} we introduce the necessary tools
and establish some basic properties of the Robin Laplacians in polygons.
Section~\ref{spl} is devoted to the proof of Theorem~\ref{thm1}.
In Section~\ref{rmk} we discuss possible generalizations and variants.
In Appendix~\ref{1d} we study the one-dimensional Robin problem which is used to
obtain a more precise result for the case $\omega=\frac\pi 2$.

\begin{figure}[t]
\centering
\myimage{40mm}{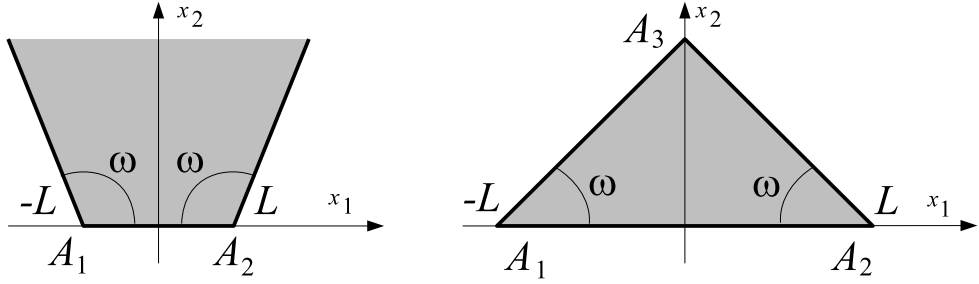}
\caption{The domain $\Omega_L$ for $\omega\ge\dfrac{\pi}{2}$ (left)
and $\omega<\dfrac{\pi}{2}$ (right).}\label{fig2}
\end{figure}

\bigskip

\noindent {\bf Acknowledgments.} The research was partially supported by ANR NOSEVOL and GDR Dynamique quantique. Bernard Helffer is also associated with the laboratoire Jean Leray at the university of Nantes.

\section{Preliminaries}\label{prel}

\subsection{Basic tools in  functional analysis}\label{ss21}

Recall the max-min principle for the self-adjoint operators.

\begin{prop}\label{maxmin}
Let $A$ be a lower semibounded self-adjoint operator
in a Hilbert space $\cH$, and let $E:=\inf\spec_\mathrm{ess} A$. For $n\in \NN$ consider the quantities
\[
E_n:=\sup_{\psi_1,\dots,\psi_{n-1}\in\cH} \inf_{\substack{u\in D(A),\, u\ne 0\\u\perp \psi_1,\dots,\psi_{n-1}}}\dfrac{\langle u,Au\rangle}{\langle u, u\rangle}.
\]
If $E_n<E$, then $E_n$ is the $n$th eigenvalue of $A$ (if numbered in the non-decreasing order
and counted with multiplicities).
Furthermore, one obtains an equivalent definition of $E_n$ by setting
\[
E_n:=\sup_{\psi_1,\dots,\psi_{n-1}\in\cH} \inf_{\substack{u\in Q(A),\, u\ne 0\\u\perp \psi_1,\dots,\psi_{n-1}}}\dfrac{a(u,u)}{\langle u, u\rangle},
\]
where $Q(A)$ is the form domain of $A$ and $a$ is the associated bilinear form.
\end{prop}

Let $\cH$ be a Hilbert space. For a closed subspace $L$ of $\cH$, we denote by $P_L$
the orthogonal projector on $L$ in $\cH$. For an ordered pair $(E,F)$ of closed subspaces
$E$ and $F$ of $\cH$ we define
\[
d(E,F)=\|P_E-P_FP_E\|\equiv \|P_E-P_EP_F\|.
\]
The following proposition summarizes some essential properties, cf. \cite[Lemma 1.3 and Proposition 1.4]{hs1}:
\begin{prop}\label{prop8}
The distance between subspaces has the following properties:
\begin{enumerate}
\item $d(E,F)=0$ if and only if $E\subset F$,
\item $d(E,G)\le d(E,F)+d(E,G)$ for any closed subspace $G$ of $\cH$,
\item if $d(E,F)<1$, then then the map $E\ni f\mapsto P_F f\in F$ is injective, and the map
$F\ni f\mapsto P_E f\in E$ has a continuous right inverse,
\item If $d(E,F)<1$ and $d(F,E)<1$, then $d(E,F)=d(F,E)$, the map
$F\ni f\mapsto P_E f\in E$ is bijective, and its inverse is continuous.
\end{enumerate}
\end{prop}

The following proposition can be used to estimate $d(E,F)$, see e.g.~\cite[Proposition 3.5]{hs1}.

\begin{prop}\label{prop9}
Let $A$ be a self-adjoint operator in $\cH$, $I\subset \RR$ be a compact interval,
$\psi_1,\dots,\psi_n\in D(A)$ be linearly independent, and $\mu_1,\dots,\mu_n\in\RR\,$.
Denote:
\begin{align*}
\varepsilon&:=\max_{j\in\{1,\dots,n\}} \big\|(A-\mu_j)\psi_j\big \|\,,\\
a&:=\frac12 \dist\big(I,(\spec A)\setminus I\big)\,,\\[\medskipamount]
\Lambda&:= \text{ the smallest eigenvalue of the Gramian matrix } \big(\langle\psi_j,\psi_k\rangle\big)\,.
\end{align*}
Let $E$ be the subspace spanned by $\psi_1,\dots,\psi_n$ and $F$ be the spectral subspace
associated with $A$ and $I$. If $a>0\,$, then
\begin{equation}
    \label{eq-def}
d(E,F)\le \dfrac{1}{a}\sqrt{\dfrac{n}{\Lambda}}\,\varepsilon\,.
\end{equation}
\end{prop}

\subsection{Robin Laplacians in infinite sectors}\label{ss-sa}

For $\alpha \in(0,\pi)$, we define 
\[
S_{\alpha}:=\big\{
(x_1,x_2)\in\RR^2: \quad \big|\arg(x_1+ix_2)\big|<\alpha
\big\}
\]
and consider the associated Robin Laplacian and the bottom of its spectrum:
\[
H_{\alpha}=H(S_\alpha,\beta)\,, \quad
E_{\alpha}:=\inf\spec H_{\alpha}\,.
\]
The following result is essentially contained in~\cite{lp}:
\begin{prop}\label{prop-seceig}
The operator $H_\alpha$ has the following properties:
\begin{itemize}
\item If $\alpha<\frac{\pi}{2}$, then
\begin{equation}\label{defEalpha}
E_{\alpha}=-\dfrac{\beta^2}{\sin^2\alpha},
\end{equation}
and this point is a simple isolated eigenvalue of $\spec H_{\alpha}$
with the associated normalized eigenfunction
\begin{equation}\label{defualpha}
U_\alpha(x_1,x_2)=\beta\sqrt{\dfrac{2\cos\alpha}{\sin^3\alpha}}\exp\Big(-\dfrac{\beta}{\sin\alpha}x_1\Big)\,.
\end{equation}
\item If $\alpha\ge\frac\pi 2$, then $E_{\alpha}=-\beta^2$ and $\spec H_{\alpha}=[E_{\alpha},+\infty)\,$.
\end{itemize}
\end{prop}
In what follows we will use another associated quantity,
\begin{equation}
      \label{eq-lba}
\Lambda_{\alpha}:=\inf (\spec H_{\alpha})\setminus\{E_{\alpha}\}.
\end{equation}
In view of Proposition~\ref{prop-seceig} we have:
\begin{itemize}
\item if $\alpha<\frac\pi 2$, then $\Lambda_{\alpha}>E_{\alpha}$. In this case, if
one denotes by $P_{\alpha}$ the orthogonal projection in $L^2(S_\alpha)$
onto the subspace spanned by $U_{\alpha}$, then the spectral theorem implies
\begin{equation} \label{eq-hproj}
\langle u, H_\alpha u\rangle\ge \Lambda_{\alpha}\|u\|^2+(E_{\alpha}-\Lambda_{\alpha})\big\langle u, P_{\alpha} u \big\rangle
\text{ for all } u\in D(H_{\alpha})\,,
\end{equation}
\item if $\alpha\ge\frac\pi 2$, then $\Lambda_{\alpha}=E_{\alpha}\,$.
\end{itemize}

\subsection{Robin Laplacians in convex polygons}\label{poly}

In this subsection, let $\Omega_1\subset\RR^2$ be a convex polygonal domain, i.e. is the intersection 
of finitely many half-planes. Assume that $\Omega_1$ has $N$ vertices
$B_1,\dots, B_N$, and the corner opening at $B_j$ will be denoted by $2\alpha_j$.
We assume that all vertices are non-trivial, which means, due to the convexity, that
 $\alpha_j\in (0,\frac\pi2)$ for all $j$.
Define 
\[
\alpha:=\min_j\alpha_j\,.
\]
Furthermore, we set $\Omega_L:=L\Omega_1$ for  some $L>0$ and denote by  $A_j:=LB_j$
the vertices of $\Omega_L $. We omit sometimes the reference to $L$ and write more  simply $\Omega\,$.
 Finally, let us pick some $\beta>0$ and consider the associated Robin Laplacian $H:=H(\Omega,\beta)$.
Strictly speaking, $H$ is the operator associated with the bilinear form
\[
h_{\Omega,\beta}(u,u)=\iint_\Omega |\nabla u|^2\,dx-\beta \int_{\partial\Omega}|u|^2\,ds\,,
\quad u\in H^1(\Omega)\,,
\]
where $ds$ means the integration with respect to the length parameter.
Using the standard methods we have 
\[
\inf\spec_\text{ess} H\ge -\beta^2\,.
\]
The following proposition is a particular case of a more general result proved in~\cite{lp}:
\begin{prop}\label{prop-lp}
$\lim_{L\to+\infty} \inf\spec H= -\dfrac{\beta^2}{\sin^2\alpha}\equiv E_\alpha$.
\end{prop}
To describe the domain of $H$, let us
recall first the Green-Riemann formula, which states that,
for $f\in H^1(\Omega)$ and $g\in H^2(\Omega)$, 
\begin{equation}
    \label{eq-gr}
\int_{\partial\Omega} f\, \dfrac{\partial g}{\partial n}\, ds
=\iint_{\Omega}\Big( f\Delta g+\nabla f\cdot\nabla g\Big)\,dx\,,
\end{equation}
where $n$ is the outward unit normal.

\begin{prop}\label{prop-dom}
There holds
\begin{equation}
     \label{eq-dh}
D(H)=\big\{u\in H^2(\Omega): \dfrac{\partial u}{\partial n}=\beta u \text{ at } \partial\Omega\big\}
\end{equation}
and $Hu=-\Delta u$ for all $u\in D(H)$.
\end{prop}

\begin{proof}
The claim follows from the general scheme developped for boundary value problems
in non-smooth domains~\cite{grisv}. We just explain briefly how this scheme appplies to the Robin
boundary condition. We note first that the associated form $h_{\Omega,\beta}$ is semibounded from below and closed
due to the standard Sobolev embedding theorems. We note then that for
any $u \in D(H)$ one  has $Hu=-\Delta u$ in  $\mathcal{D}'(\Omega)$.
Furthermore, if $\Tilde D$ is the set on the right-hand
side of \eqref{eq-dh}, then it easily follows from \eqref{eq-gr} that $\Tilde D\subset D(H)$.
It follows also that for $f\in H^2(\Omega)$ the inclusion
$f\in D(H)$ is equivalent to the equality $\partial f/\partial n=\beta f$ on $\partial\Omega\,$.
In view of these observations, it is sufficient to show that $D(H)\subset H^2(\Omega)$.

Take any $f\in D(H)\subset H^1(\Omega)$ and let $g:=H f\in L^2(\Omega)$.
All corners at the boundary of $\Omega$
are smaller than $\pi$, and the trace of $f$ on 
$\partial \Omega$ is in $H^\frac 12(\partial \Omega)$,
which means that there exists a solution $u\in H^2(\Omega)$
for the  boundary value problem:
\[
-\Delta u= g \text{ in }\Omega, \, \dfrac{\partial u}{\partial n}=\beta f \text{ on }\partial\Omega\,,
\]
see \cite[Section 2.4]{grisv} (we are in the case where no singular solutions are present).
On the other hand, $f$ is a variational solution
of the preceding problem. This means that the function $v:=f-u\in H^1(\Omega)$ becomes 
a variational solution to
\[
-\Delta v= 0 \text{ in } \mathcal{D}'(\Omega), \, \dfrac{\partial v}{\partial n}= 0
\text{ on }\partial\Omega\,.
\]
Again according to \cite[Section 2.4]{grisv} we conclude that the only possible solution
is constant, which means that $f=u+v\in H^2(\Omega)$.
\end{proof}

Now let us obtain some (Agmon-type) decay estimates of the eigenfunctions of $H$
corresponding to the lowest eigenvalues as $L\to +\infty$. Let us start with a technical
identity.

\begin{lemma}\label{lem8}
Let $u\in H^2(\Omega)$ be real-valued and satisfy the Robin boundary condition $\partial u/\partial n=
\beta u$ at $\partial \Omega$.
Furthermore, let
$\Phi:\Omega\to\RR$ be such that $\Phi,\nabla\Phi\in L^\infty(\Omega)$, then
\[
  \iint_{\Omega} \big|\nabla (e^\Phi u)\big|^2dx -\beta \int_{\partial\Omega} e^{2\Phi}u^2ds
  =\iint_{\Omega} e^{2\Phi} u(-\Delta u)dx+\iint_{\Omega} |\nabla\Phi|^2 e^{2\Phi}u^2dx.
\]
\end{lemma}

\begin{proof}
We just consider the case $\Phi\in C^2(\Bar \Omega)$, then one can pass to the general case
using the standard regularization procedure.
We have
\begin{align*}
\big|\nabla(e^\Phi u)\big|^2 & = \Big(\dfrac{\partial}{\partial x_1}(e^\Phi u)\Big)^2
+
\Big(\dfrac{\partial}{\partial x_2}(e^\Phi u)\Big)^2\\
&=
\Big(\dfrac{\partial\Phi}{\partial x_1} e^\Phi u + e^\Phi \dfrac{\partial u}{\partial x_1}\Big)^2
+
\Big(\dfrac{\partial\Phi}{\partial x_2} e^\Phi u + e^\Phi \dfrac{\partial u}{\partial x_2}\Big)^2\\
&= |\nabla\Phi|^2e^{2\Phi} u^2 + 2e^\Phi u \nabla \Phi\cdot\nabla u+e^{2\Phi} |\nabla u|^2\\
&=|\nabla \Phi|^2e^{2\Phi} u^2+\nabla(e^{2\Phi} u)\cdot \nabla u\,.
\end{align*}
Integrating this equality in  $\Omega$,  we arrive at
\begin{multline*}
\iint_{\Omega} \big|\nabla(e^\Phi u)\big |^2 dx
=
\iint_\Omega| \nabla \Phi|^2e^{2\Phi} u^2dx + \iint_\Omega \nabla(e^{2\Phi} u)\cdot \nabla u dx\\
=
\iint_\Omega |\nabla \Phi|^2e^{2\Phi} u^2dx +\int_{\partial\Omega} e^{2\Phi} u \dfrac{\partial u}{\partial n}ds
+\iint_\Omega e^{2\Phi} u(-\Delta u)dx
\\
=
\iint_\Omega |\nabla \Phi|^2e^{2\Phi} u^2dx +\beta\int_{\partial\Omega} e^{2\Phi} u^2ds +\iint_\Omega e^{2\Phi} u(-\Delta u)\,dx\,.
\end{multline*}
\end{proof}

Now, let us choose a constant $b>0$ such that all corners of $\Omega$ are contained
in the ball of radius $bL$ centered at the origin, and consider the function
$\Phi:\Omega\to \RR$ defined by 
\[
\Phi(x):=\beta \min\Big\{\min_{j\in\{1,\dots,N\}}\cot\alpha_j\cdot |x-A_j|, b L\Big\}.
\]
For a compact $\Omega$ we choose the constant $b$ sufficiently large, so that
the exterior minimum can be dropped.

\begin{prop}\label{lem-agmon}
Let $\lambda=\lambda(L)>0$ be such that 
\[
\lim_{L\to+\infty} \lambda(L)=0\,.
\]
Then, for any $\varepsilon\in(0,1)$ there exists $C_\varepsilon>0$  and $L_\varepsilon$ such that, 
if $E=E(L)$ is an eigenvalue of $H$ satisfying
\begin{equation}
        \label{eq-leq}
E\le -\,\dfrac{\beta^2}{\sin^2 \alpha}+\lambda\,,
\end{equation}
and $u$ is an associated normalized eigenfunction, then
\[
\big\|e^{(1-\varepsilon)\Phi} u\big\|_{H^1(\Omega)}\le C_\varepsilon e^{\varepsilon L} \quad \text{for} \quad L\geq L_\epsilon\,.
\]
\end{prop}

\begin{proof}
Let $r>0$. Let us pick a $C^\infty$ function $\chi:[0,+\infty)\to[0,1]$ such that
$\chi(t)=1$ for $t\le r$ and $\chi(t)=0$ for $t>2r$, and
introduce 
\[
\Tilde \chi_{j}(x)=\chi\Big(\dfrac{|x-A_j|}{L}\Big), \quad j=1,\dots ,N\,.
\]
We assume that $r$ is sufficiently small, which ensures that the supports of $\Tilde\chi_j$
are disjoint and that $\Phi(x)=\beta\cot\alpha_j\,  |x-A_j|$ for $x\in\supp \Tilde\chi_j$.
An exact value of $r$ will be chosen later. We also complete by the function 
\[
\Tilde\chi_{0}:=1-\sum_{j=1}^N \Tilde \chi_j \,,
\]
and, finally, set
\[
\chi_j:=\Tilde \chi_j \Big/ \sqrt{\sum_{k=0}^N \Tilde\chi_k^2}, \quad j=0,\dots, N.
\]
We observe  that we have the equalities $\supp\chi_j=\supp\Tilde\chi_j$, that
each $\chi_j$ is $C^\infty$, and that
\[
\sum_{j=0}^N \chi_j^2=1\,.
\]

For any $v\in H^1(\Omega)$ we also have $\chi_j v\in H^1(\Omega)$, and
by a direct computation one obtains
\[
h_{\Omega,\beta}(v,v)=\sum_{j=0}^N h_{\Omega,\beta}(\chi_j v,\chi_j v)-
\sum_{j=0}^N\big\| v \nabla\chi_j\big\|^2\,.
\]
By construction of $\chi_{j}$, we one can find a constant $C>0$ independent of $v$ and $L$ with
\[
h_{\Omega,\beta}(v,v)\ge\sum_{j=0}^N h_{\Omega,\beta}(\chi_j v,\chi_j v)-\dfrac{C}{L^2}\|v\|^2
\text{ for large } L\,.
\]
Now let us denote $\Psi:=(1-\varepsilon)\Phi$.  By applying the preceding inequality we obtain
\begin{multline}
         \label{eq-i0}
I:=\iint_{\Omega}\big|\nabla(e^\Psi u)\big|^2dx - \beta\int_{\partial\Omega} |e^\Psi u|^2ds
\ge \delta \iint_{\Omega}\big|\nabla(e^\Psi u)\big|^2\,dx
\\+
(1-\delta)\bigg[
\sum_{j=0}^N \Big(\iint_{\Omega}\big|\nabla(\chi_j e^\Psi u)\big|^2\, dx
-\dfrac{\beta}{1-\delta}\int_{\partial\Omega }\big|\chi_j e^\Psi u\big|^2\,ds
\Big)
-\dfrac{C}{L^2}\iint_{\Omega} |e^\Psi u|^2\,dx\,,
\bigg],
\end{multline}
where $\delta\in(0,1)$ is a constant which will be chosen later.

Furthermore, considering $\chi_j e^\Psi u$ as a function from $H^1(S_j)$, where
$S_j$ is a suitably rotated copy of the sector $S_{\alpha_j}$ (see Subsection~\ref{ss-sa})
which coincides with $\Omega$ near $A_j$,  we have, for  $j=1,\dots,N$,
\[
\iint_{\Omega}\big|\nabla(\chi_j e^\Psi u)\big|^2dx
-\dfrac{\beta}{1-\delta}
\int_{\partial\Omega }\big|\chi_j e^\Psi u\big|^2ds
\ge 
-\dfrac{\beta^2}{(1-\delta)^2\sin^2\alpha_j}\iint_{\Omega} |\chi_j e^\Psi u|^2dx.
\]

By the preceding constructions, the support of $\chi_0$ is of the form
 $\supp \chi_0=L \Omega'$ with some $L$-independent $\Omega'$.
Furthermore, one can construct a smooth domain $D$ with  $L\Omega'\subset LD\subset \Omega$
and such that $\partial (L\Omega')\cap \partial\Omega=\partial (L D)\cap \partial\Omega$.
As mentioned in the introduction, the lowest eigenvalue of
$H\big(LD,\beta/(1-\delta)\big)$ for large $L$ converges to $-\beta^2/(1-\delta)^2$,
i.e. for any $v\in H^1(LD)$ we have
\[
\iint_{LD}\big|\nabla v\big|^2dx
-\dfrac{\beta}{1-\delta}
\int_{\partial(LD) }| v|^2ds
\ge -\Big(\dfrac{\beta^2}{(1-\delta)^2}+\varepsilon_0\Big)
\iint_{LD} |v|^2dx, 
\]
where $\varepsilon_0:=\varepsilon_0(L,\delta)>0$ is such that
$\lim_{L\to+\infty}\varepsilon_0=0$ for any fixed $\delta\in(0,1)$.
By taking $v=\chi_0 e^\Psi u$ 
we obtain
\[
\iint_{\Omega}\big|\nabla(\chi_0 e^\Psi u)\big|^2dx
-\dfrac{\beta}{1-\delta}
\int_{\partial\Omega }\big|\chi_0 e^\Psi u\big|^2ds
\ge -\Big(\dfrac{\beta^2}{(1-\delta)^2}+\varepsilon_0\Big)
\iint_{\Omega} |\chi_0 e^\Psi u|^2dx.
\]

Putting the preceding estimates together we arrive at
\begin{multline}
       \label{eq-i1}
I\ge \delta \iint_{\Omega}\big|\nabla(e^\Psi u)\big|^2dx
-\Big( \dfrac{\beta^2}{1-\delta} + \dfrac{(1-\delta)C}{L^2}+\varepsilon_1\Big)
\iint_{\Omega}|\chi_0e^\Psi u|^2dx\\
-
\sum_{j=1}^N
\Big(
\dfrac{\beta^2}{(1-\delta)\sin^2\alpha_j}
+\dfrac{(1-\delta)C}{L^2}
\Big)
\iint_{\Omega}|\chi_je^\Psi u|^2dx
\end{multline}
with $\varepsilon_1:=(1-\delta)\varepsilon_0$. 
On the other hand, due to Lemma~\ref{lem8} we have
\begin{multline}
   \label{eq-i2}
I=\iint_{\Omega} e^{2\Psi} u(-\Delta u)dx+\iint_{\Omega} |\nabla \Psi|^2 e^{2\Psi}u^2\,dx\\
=E \iint_{\Omega} e^{2\Psi} u^2 dx+\iint_{\Omega} |\nabla \Psi|^2 e^{2\Psi}u^2\,dx
=\sum_{j=0}^N \iint_{\Omega}\big(E+ |\nabla \Psi|^2 \big)|\chi_j e^\Psi u|^2\,dx\,.
\end{multline}
We estimate as follows:
\begin{gather*}
\big|\nabla \Psi(x)\big|\le (1-\varepsilon)^2\beta^2 \cot\alpha\equiv
(1-\varepsilon)^2\beta^2\Big(\dfrac{1}{\sin^2\alpha}-1\Big), \quad x\in\supp\chi_0\,,\\
\big|\nabla \Psi(x)\big|\le (1-\varepsilon)^2\beta^2 \cot\alpha_j\equiv
(1-\varepsilon)^2\beta^2\Big(\dfrac{1}{\sin^2\alpha_j}-1\Big), \quad x\in\supp\chi_j\,,
\quad j=1,\dots,N\,.
\end{gather*}
Substituting these two inequalities into \eqref{eq-i2} and using \eqref{eq-leq} we arrive at
\begin{align*}
I&\le \bigg( -\dfrac{\beta^2}{\sin^2\alpha}+\lambda +(1-\varepsilon)^2 \beta^2
\Big(\dfrac{1}{\sin^2\alpha}-1\Big)
\bigg)\iint_{\Omega} |\chi_0 e^\Psi u|^2 dx\\
&\quad +\sum_{j=1}^N
\bigg( -\dfrac{\beta^2}{\sin^2\alpha}+\lambda +(1-\varepsilon)^2 \beta^2
\Big(\dfrac{1}{\sin^2\alpha_j}-1\Big)
\bigg)
\iint_{\Omega} |\chi_j e^\Psi u|^2\, dx\,.
\end{align*}
Combining with \eqref{eq-i1} we have:
\begin{gather*}
\delta\iint_{\Omega}\big|\nabla(e^\Psi u)\big|^2\,dx
+C_0 \iint_{\Omega} |\chi_0e^\Psi u|^2\,dx \le 
\sum_{j=1}^N C_j |\chi_je^\Psi u|^2\,dx\,,\\
\intertext{where}
\begin{aligned}
C_0&:= (2\varepsilon-\varepsilon^2)\Big(\dfrac{1}{\sin^2\alpha}-1\Big)\,\beta^2 -\dfrac{\delta}{1-\delta}\,\beta^2
-\dfrac{(1-\delta)C}{L^2}-\varepsilon_1-\lambda\,,\\
C_j&:=-\dfrac{\beta^2}{\sin^2\alpha}+(1-\varepsilon)^2\Big(\dfrac{1}{\sin^2 \alpha_j}-1\Big)\beta^2
+\dfrac{\beta^2}{(1-\delta)\sin^2\alpha_j}+\dfrac{(1-\delta)C}{L^2}+\lambda\,,\, j=1,\dots,N\,.
\end{aligned}
\end{gather*}
As $\varepsilon>0$ is a fixed positive number and both $\varepsilon_1$ and $\lambda$ tend to $0$ as $L\to +\infty\,$, we can find
$m_\varepsilon>0$, $\delta>0$ and $L_0>0$ such that $C_0\ge m_\varepsilon$
for all  $L>L_0\,$. At the same time, for the same $\delta$ and $L$
we may estimate $C_j\le M_{\varepsilon}\,$, $j=1,\dots,N\,$, which gives
\[
\iint_{\Omega}\big|\nabla(e^\Psi u)\big|^2\,dx
+\iint_{\Omega} |\chi_0e^\Psi u|^2\,dx \le C_{\varepsilon}
\sum_{j=1}^N |\chi_je^\Psi u|^2\,dx\,, \quad
C_{\varepsilon}:=\dfrac{M_{\varepsilon}}{\delta}+\dfrac{M_{\varepsilon}}{m_\varepsilon}\,.
\]
Now we get the estimate
\begin{multline*}
\|e^{(1-\varepsilon)\Phi}u\|^2_{H^1(\Omega)}
=\|e^\Psi u\|^2_{H^1(\Omega)}=
\iint_\Omega \big|\nabla(e^\Psi u)\big|^2\,dx
+\iint_\Omega |e^\Psi u|^2\,dx\\
=
\iint_\Omega \big|\nabla(e^\Psi u)\big|^2\,dx
+\iint_\Omega |\chi_0 e^\Psi u|^2\,dx
+\sum_{j=1}^N |\chi_je^\Psi u|^2\,dx
\le (1+C_{\varepsilon})\sum_{j=1}^N |\chi_je^\Psi u|^2\,dx\\
\le (1+C_{\varepsilon}) \exp\Big[(1-\varepsilon)\max_{j\in\{1,\dots,N\}}\sup_{x\in\supp\chi_j} \Phi(x)\Big]
\sum_{j=1}^N \iint_{\Omega}|\chi_j u|^2\,dx\,.
\end{multline*}
We have
\[
\sum_{j=1}^N \iint_{\Omega}|\chi_j u|^2\,dx
\le \sum_{j=0}^N \iint_{\Omega}|\chi_j u|^2\,dx= \iint_{\Omega}|u|^2\,dx=1\,,
\]
and 
\[
\max_{j\in\{1,\dots,N\}}\sup_{x\in\supp\chi_j} \Phi(x)\le 2r\beta(\cot\alpha) L\,.
\]
Therefore, by taking $r< \varepsilon/(2t\beta\cot\alpha)\,$,  we get the conclusion.
\end{proof}

\section{The lowest eigenvalues of $H_L$}\label{spl}

\subsection{Notation}

In this section we study in greater detail the lowest
eigenvalues of the operator $H_L$. We collect first some notation and conventions
used below. Note that all the assertions of Section~\ref{prel} are applicable to $H_L$ as well.
Throughout the section we will write 
\[
\alpha:=\dfrac \omega 2 \quad \mbox{ and } \quad \Omega:=\Omega_L\,.
\]
Furthermore, we introduce the following transformations of $\RR^2\,$:
\[
R_1 (x_1,x_2)=\begin{pmatrix}
\cos \alpha & \sin\alpha\\
-\sin\alpha & \cos\alpha
\end{pmatrix}
\begin{pmatrix}
x_1+L\\ x_2
\end{pmatrix},
\quad
R_2 (x_1,x_2)=\begin{pmatrix}
\cos \alpha & \sin\alpha\\
\sin\alpha & -\cos\alpha
\end{pmatrix}
\begin{pmatrix}
L-x_1\\ x_2
\end{pmatrix}.
\]
The geometric meaning of $R_j$ is clear from the equalities
$R_j (\Sigma_j)=S_\alpha$, $j=1,2$,
and we consider the associated rotated eigenfunctions
\[
U_j(x):=U_{\alpha}(R_j x)\,,  \, j=1,2\,.
\]
Recall that $S_\alpha$ and $U_\alpha$ are defined in Subsection~\ref{ss-sa}, so we have
\begin{equation}
   \label{eq-UUU}
\begin{aligned}
U_1(x_1,x_2)&=\beta \sqrt{\dfrac{2\cos\alpha}{\sin^3\alpha}} \,e^{-\beta (x_1+L)\cot\alpha -\beta x_2},\\
U_2(x_1,x_2)&=\beta \sqrt{\dfrac{2\cos\alpha}{\sin^3\alpha}} \, e^{-\beta (L-x_1)\cot\alpha -\beta x_2}.
\end{aligned}
\end{equation}
We also recall the notation 
\[
E_\alpha:=-\beta^2/ \sin^2\alpha\,.
\]
Furthermore, for $j=1,2$ we denote by $M_j$ the Robin Laplacian in $\Sigma_j$,
\[
M_j:=H(\Sigma_j,\beta)\,.
\]

\subsection{A rough eigenvalue estimate}

Let us obtain some rough information on the behavior of the eigenvalues  of
$H_L$ as $L$ tends to $+\infty$. Assuming that $H_L$ has at least $n-1$ eigenvalues below the essential spectrum, we
denote
\[
\widetilde E_n (L):=\inf (\spec H_L)\setminus\big\{E_1(L),\dots,E_{n-1}(L)\big\}\,,
\]

\begin{lemma}\label{lem10}
Let $\omega\in \big(0,\frac \pi 3\big) \mathbin{\cup} \big[\frac\pi 2,\pi\big)$, then
for sufficiently large $L$ the operator $H_L$ has at least two eigenvalues below the essential spectrum,
and one has
\begin{gather}
     \label{eq-e12}
\lim_{L\to+\infty} E_j(L)=E_{\alpha}, \quad j=1,2\,,\\
       \label{eq-e31}
\liminf_{L\to+\infty} \Tilde E_3(L) > E_{\alpha}\,.
\end{gather}
\end{lemma}

\begin{proof}
For $\delta>0$, let us pick a $C^\infty$ function $\chi:\RR_+\to[0,1]$ such that
$\chi(t)=1$ for $t \le \delta$ and $\chi(t)=0$ for $t>2\delta\,$.
Introduce the functions
\[
\Tilde\chi_{j}(x)=\chi\Big(\dfrac{|x-A_j|}{L}\Big), \quad j=1,2\,.
\]
We assume that $\delta$ is sufficiently small, which ensures that the supports of $\Tilde\chi_1$
and $\Tilde\chi_2$ do not intersect, and consider the functions
\[
v_j:=\Tilde\chi_{j} U_j\,,\quad j=1,2.
\]
By a simple computation,
as $L\to +\infty$ we have
\[
\iint_{\Omega} v_j v_k \,dx=\delta_{jk}+o(1), \quad
\iint_{\Omega} \nabla v_j\cdot \nabla v_k \,dx-\beta\int_{\partial\Omega}  v_j v_k\, ds=
E_{\alpha}\delta_{jk}+o(1)\,, \quad j,k=1,2.
\]
It follows that
\[
\sup_{0\not\equiv v\in\linspan(v_1,v_2)}
\dfrac{h_{\Omega,\beta}(v,v)}{\langle v,v\rangle}
\le  E_\alpha+o(1)< -\beta^2\equiv\inf\spec_\text{ess} H_L\,,
\]
the last inequality being true for $L$ large enough.

On the other hand, the functions $v_1$ and $v_2$ are linearly independent. It follows
that for any $\psi\in L^2(\Omega)$ one can find
a non-trivial linear combination $v\in\linspan(v_1,v_2)$ which is orthogonal to $\psi$.
Due to the previous estimate and Proposition~\ref{maxmin}
we obtain then
\[
E_2(L)\le E_\alpha+o(1)\,.
\]
Combining with $E_2(L)\ge E_1(L)\,$,
and with the result of Proposition~\ref{prop-lp}, this gives \eqref{eq-e12}.

Let us now prove \eqref{eq-e31}.  Let us introduce
\[
\Tilde \chi_{0}:=1-\Tilde\chi_1-\Tilde\chi_2
\]
and set
\[
\chi_j:=\Tilde \chi_j \Big/ \sqrt{\sum_{k=0}^2 \Tilde\chi_k^2}, \quad j=0,1, 2.
\]
%color
By a direct computation, for any $u\in H^1(\Omega)$ we have
\[
h_{\Omega,\beta}(u,u)=\sum_{j=0}^2 h_{\Omega,\beta}(\chi_j u,\chi_j u)-
\sum_{j=0}^2\big\| u \nabla\chi_j\big\|^2,
\]
and by the construction of $\chi_{j}$, we can find $L_0>0$ and  $C>0$ such that for all  $u$ and $L\geq L_0$
\[
h_{\Omega,\beta}(u,u)\ge\sum_{j=0}^2h_{\Omega,\beta}(\chi_j u)-\dfrac{C}{L^2}\|u\|^2\,.
\]
Furthermore, we have $\chi_j u\in H^1(\Sigma_j)\,$, $j=1,2\,$.
Consider the orthogonal projections $\Pi_j:=\langle U_j,\cdot\rangle U_j$ in $L^2(\Sigma_j)\,$.
By applying the inequality \eqref{eq-hproj} we obtain
\[
h_{\Omega,\beta}(\chi_j u,\chi_j u)\ge (E_\alpha-\Lambda_\alpha) \|\Pi_j \chi_j u \|^2_{L^2(\Sigma_j)}
+\Lambda_{\alpha}\|\chi_j u\|^2_{L^2(\Sigma_j)}, \quad j=1,2\,.
\]
The norms in $L^2(\Sigma_j)$ can be replaced back by the norms in $L^2(\Omega)$, and we infer
\[
h_{\Omega,\beta}(u,u)\ge \langle u,\Pi u\rangle +
\Lambda_\alpha\big(\|\chi_1 u\|^2+\|\chi_2 u\|^2\big)
+ h_{\Omega,\beta}(\chi_0 u,\chi_0 u) -\dfrac{C}{L^2}\|u\big\|^2\,,
\]
where $\Pi:=(E_\alpha-\Lambda_\alpha)\big(\chi_1\Pi_1 \chi_1 +\chi_2\Pi_2 \chi_2\big)$
is an operator whose range is at most two-dimensional.

To estimate the term with $\chi_0$, we proceed as in the proof of Proposition~\ref{lem-agmon}.
By the preceding constructions, the support of $\chi_0$
has the form $\supp \chi_0=L \Omega'$ with some $L$-independent $\Omega'$.
Furthermore, one can construct a convex polygonal domain $D$ with  $L\Omega'\subset LD\subset \Omega$
such that $\partial (L\Omega')\cap \partial\Omega=\partial (L D)\cap \partial\Omega$
and that the minimal corner $\theta$ at the boundary of $D$ is strictly larger than $\omega$.
By Proposition~\ref{prop-lp} for any $A<E_{\theta/2}$ and any 
$v\in H^1(LD)$ we have, as $L$ is sufficiently large,
\[
h_{LD,\beta}(v,v)\ge  A \|v\|^2_{L^2(LD)}.
\]
As $E_{\theta/2}>E_{\omega/2}\equiv E_\alpha$, we may assume that $A>E_\alpha$.
Using the last equality with $v=\chi_0 u$ we obtain, for large $L$,
\[
h_{\Omega,\beta}(\chi_0 u,\chi_0 u)\ge A \|\chi_0 u\|^2\,.
\]

Putting all together and noting that $\|\chi_0 u\|^2+\|\chi_1 u\|^2+\|\chi_2 u\|^2=\|u\|^2$
we obtain, for sufficiently large $L$,
\[
h_{\Omega,\beta}(u,u)\ge \langle u,\Pi u\rangle + \Big(E-\dfrac{C}{L^2}\Big)\|u\big\|^2, \quad
E=\min(A,\Lambda_\alpha)>E_\alpha\,.
\]
Now take two vectors $\psi_1$ and $\psi_2$ spanning the range of $\Pi\,$.
For any non-zero  $u\in H^1(\Omega)$ which is orthogonal to $\psi_1$ and $\psi_2$
we have
\[
\dfrac{h_{\Omega,\beta}(u,u)}{\langle u,u\rangle}\ge E-\dfrac{C}{L^2}\,,
\]
which gives the announced inequality \eqref{eq-e31} by the max-min principle.
\end{proof}

The following assertion summarizes the preceding considerations:
\begin{prop}\label{corol13}
Let $\omega\in \big(0,\frac \pi 3\big) \mathbin{\cup} \big[\frac\pi 2,\pi\big)$, then
there exists $\delta>0$ and $L_0$ such that for $L\geq L_0$ the spectrum of $H_L$
in $(E_\alpha-\delta,E_\alpha+\delta)$ consists of exactly two eigenvalues $E_1(L)$ and $E_2(L)$, both converging to $E_\alpha$ as $L\to+\infty$.
\end{prop}

\begin{remark}\label{rem-tri}
Indeed, one can prove an analog of Lemma~\ref{lem10} for the remaining ranges of $\omega$
in a similar way, and one has:
\begin{gather}
\lim_{L\to+\infty} E_1(L)=E_{\alpha} \quad \text{and} \quad \liminf_{L\to+\infty} \Tilde E_2(L) > E_{\alpha} \quad
\text{for } \omega\in\Big(\dfrac\pi 3,\dfrac\pi 2\Big),\nonumber\\
\lim_{L\to+\infty} E_j(L)=E_{\alpha},\quad j=1,2,3, \quad \text{and}\quad \liminf_{L\to+\infty} \Tilde E_4(L) > E_{\alpha} 
\text{ for } \omega=\dfrac{\pi}{3},
\label{eq-equi}
\end{gather}
and Proposition~\ref{corol13} should be suitably reformulated.
We remark that the case $\omega=\pi/3$, i.e. the equilateral triangle, was already studied
in \cite[Section~7]{mc}, where it was found that after a suitable transformation
one may separate the variables, and the calculation of the eigenvalues reduces to solving
a certain non-linear system, which admits a rather direct analysis.
In particular,  the second inequality in~\eqref{eq-equi}
holds in the stronger form $\lim_{L\to+\infty} \Tilde E_4(L) = -\beta^2$.
\end{remark}

%color

For the rest of the section,
we assume that 
\[
\omega\in \Big(0,\dfrac \pi 3\Big) \mathbin{\cup} \Big[\dfrac\pi 2,\pi\Big).
\]

\subsection{Cut-off functions}\label{ss-cutoff}

We are going to introduce a family of cut-off functions adapted to the geometry of
the sector $S_\alpha$ (see Subsection~\ref{ss-sa}). Note that our assumptions imply $\alpha <\frac \pi 2$.
Pick a function $\chi:\RR\to [0,1]$ such that
\begin{equation}
     \label{eq-chi}
\chi\in C^\infty(\RR), \quad \chi(t)=1 \text{ for } t \le-1, \quad \chi(t)=0 \text{ for } t \ge 0\,,
\end{equation}
and for $\ell>0$ we set
\begin{equation}
    \label{eq-pl}
\varphi_{\alpha,\ell}(x_1,x_2)=\chi(x_1-\ell\cos\alpha) \chi\big(|x|-(\ell-1)\big)\,.
\end{equation}

This function has the following properties for large $\ell$, see~Figure~\ref{fig3}:
\begin{equation}\label{eq-prophi}
\begin{gathered}
\varphi_{\alpha,\ell}\in C^\infty(\Bar S_\alpha)\,,\\
\varphi_{\alpha,\ell}(x)\in[0,1] \text{ for all } x\in S_\alpha\,,\\
\varphi_{\alpha,\ell}(x)=1 \text{ for } x=(x_1,x_2)\in\{x_1\le \ell\cos\alpha-2\}\cap S_\alpha\,,\\[\smallskipamount]
\varphi_{\alpha,\ell}(x)=0 \text{ for } x=(x_1,x_2)\notin \{x_1\le \ell\}\cap S_\alpha\,,\\
\dfrac{\partial\varphi_{\alpha,\ell}}{\partial n}=0 \text{ at } \partial S_\alpha\,, \\%[\medskipamount]
\text{$\sum_{|\nu|\le2}\|D^\nu \varphi_{\alpha,\ell}\|_\infty\le c$
for some $c >0$ independent of $\ell\,$.}
\end{gathered}
\end{equation}
The slightly involved construction of $\varphi_{\alpha,\ell}$ guarantees
that for any function $f\in H^2(S_\alpha)$ with $\partial f/\partial n=\beta f$
at the boundary the product $\varphi_{\alpha,\ell} f$
still satisfies the same boundary condition.

\begin{figure}[t]
\centering
\myimage{70mm}{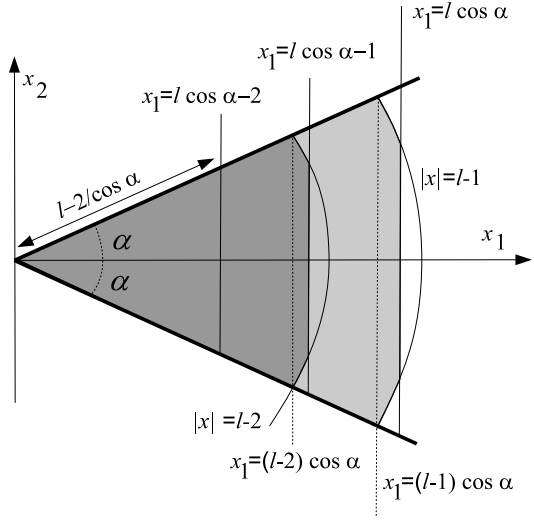}
\caption{The function $\varphi_{\alpha,\ell}$ vanishes outside the shaded domains,
and equals $1$ in the dark shaded domain.
\label{fig3}}
\end{figure}
Finally, we set
\[
\psi_{\alpha,\ell}(x):=\varphi_{\alpha,\ell}(x) U_{\alpha}(x)\,,
\]
where $U_\alpha$ is defined in \eqref{defualpha}.
Using the properties~\eqref{eq-prophi} and a simple direct computation one obtains:

\begin{lemma}\label{lem-uu}
The function $\psi_{\alpha,\ell}$ belongs to the domain of $H_\alpha\,$, and 
the following estimates are valid as $\ell\to+\infty\,$:
\begin{align}
        \label{eq-unorm}
\|\psi_{\alpha,\ell}\|^2_{L^2(S_\alpha)}&=1+O(\ell e^{-2\beta\ell\cot\alpha})\,,\\ 
        \label{eq-ulapl}
\big\|(-\Delta-E_\alpha)\psi_{\alpha,\ell}\big\|^2_{L^2(S_\alpha)}&=O(\ell e^{-2\beta\ell\cot\alpha})\,.
\end{align}
\end{lemma}

Now let us choose the maximal constant $\tau >1$ such that the two isosceles triangles
$\Theta_1(\tau L)$ and $\Theta_2(\tau L)$
with the side length $\tau L$ and the vertex angle $\omega$ spanned at the boundary of
$\Omega$ near respectively $A_1$ and $A_2$ are included in $\Omega$.
More precisely,
\begin{equation}\label{deftau}
\tau :=\begin{cases}
\dfrac{1}{\cos\omega}, & \omega\in\Big(0,\dfrac\pi 3\Big)\,,\\[\medskipamount]
2, & \omega\in \Big[\dfrac\pi 2,\pi\Big)\,
\end{cases}
\end{equation}
see Figure~\ref{fig4}.

\begin{figure}
\centering
\myimage{40mm}{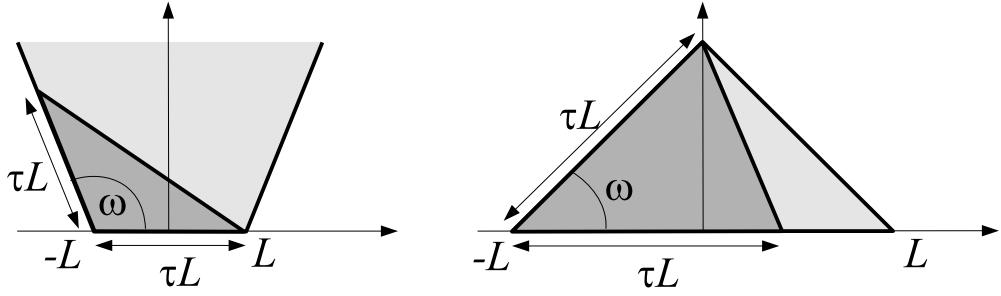}
\caption{The choice of the constant $\tau $.}\label{fig4}
\end{figure}

Consider the functions
\[
\psi_j(x)=v_j(x) U_j(x) \quad \text{with} \quad v_j(x):=\varphi_{\alpha,\tau L}(R_j x)\,,
\quad j=1,2\,.
\]

By Proposition~\ref{corol13} we can find $\delta>0$ such that
the interval $I:=(E_\alpha-\delta,E_\alpha+\delta)$
contains exactly two eigenvalues of $H_L$ and  the larger interval
$(E_\alpha-2\delta,E_\alpha+2\delta)$
does not contain any further spectrum for large $L$.

Let $E$ denote the subspace spanned by $\psi_j$, $j=1,2$,
and $F$ denote the spectral subspace of $H_L$ corresponding to $I\,$.
We are going to estimate the distances $d(E,F)$ and $d(F,E)$ between these two subspaces,
see Subsection~\ref{ss21}.

\begin{lemma}\label{gram}
For the Gramian matrix  $G:=(g_{jk})=\big(\langle \psi_j,\psi_k\rangle\big)$ we have
\[
g_{jk}= \delta_{jk} +O(L e^{-2\beta L\cot\alpha})\,, \quad j,k=1,2.
\]
 Furthermore, $g_{11}=g_{22}$
and $g_{12}=g_{21}\,$.
\end{lemma}

\begin{proof}
The identities for the coefficients follow from the considerations of symmetry.
It follows from Lemma~\ref{lem-uu} that
\[
\|\psi_j\|^2=1+O(Le^{-2\tau \beta L\cot\alpha})\quad \text{ for }  j=1,2\,.
\]
On the other hand, using the explicit expressions~\eqref{eq-UUU} for $U_j$\,,
we obtain
\[
\psi_1(x_1,x_2)\psi_2(x_1,x_2)=2\beta^2\dfrac{\cos\alpha}{\sin^3\alpha}
\varphi_{\alpha,\tau L}(R_1 x)\varphi_{\alpha,\tau L}(R_2 x)
\exp\big( -2\beta L\cot\alpha) \exp( - 2\beta x_2\big)\,.
\]
Using the properties~\eqref{eq-prophi} we have
\[
\langle \psi_1,\psi_2\rangle =O(L e^{-2\beta L\cot\alpha})\,.
\]
As $\tau >1$ by~\eqref{deftau}, this gives the result.
\end{proof}

\begin{lemma}\label{lem15}
For  large $L$ there holds 
\[
d(E,F)=d(F,E)=O(\sqrt L e^{-\beta \tau L \cot\alpha})\,.
\]
\end{lemma}

\begin{proof}
Let us show first the desired estimate for $d(E,F)$.
By Lemma~\ref{lem-uu}, we have
\[
\big\|(H_L-E_\alpha)\psi_j\big\|= O(\sqrt L e^{-\beta \tau L\cot\alpha})\,.
\]
Using Proposition~\ref{prop9} for the previously chosen interval $I$ and
applying Lemma~\ref{gram} gives the result.

We will now show that $d(F,E)<1$ for large $L$, then by Proposition~\ref{prop8}
it will follow that $d(F,E)=d(E,F)$. 

Let $\varphi:\RR\to\RR$ be a $C^\infty$ function
such that $\varphi(t)=1$ for $t$ near $0$ and $\varphi(t)=0$ for $t >\frac 1 2$ and
introduce
\[
\chi_j(x):=\varphi\Big(\dfrac{|x-A_j|}{L}\Big), \quad j=1,2\,,
\quad 
\chi_0:=1-\chi_1 - \chi_2\,.
\]
Let $u_k$ be a normalized eigenfunction of $H_L$ associated with  $E_k(L)\,$, $k=1,2$.
We know (Proposition~\ref{corol13}) that $E_k(L)$ tends to $E_\alpha$ as $L\to+\infty$, so Proposition~\ref{lem-agmon}
is applicable to $u_k$. In particular, for some $\sigma>0$ we have
\[
\|\chi_0 u_k\|_{L^2(\Omega)}=O(e^{-\sigma L})\,.
\]
 Furthermore, using Proposition~\ref{prop-dom}
we check that $\chi_j u_k\in D(H_L)$ and that
\begin{multline*}
\big\|(H_L - E_\alpha) (\chi_j u_k)\big\|_{L^2(\Omega)} =
\big\|(-\Delta - E_\alpha) (\chi_j u_k)\big\|_{L^2(\Omega)}\\
=\big\|-(\Delta \chi_j)u_k -2\nabla \chi_j\nabla u_k\big\|_{L^2(\Omega)}
=O(e^{-\sigma' L})\,,
\end{multline*}
for some $\sigma'>0\,$, and by taking the minimum we may assume that $\sigma=\sigma'\,$.
The last estimate can be also rewritten as an estimate in $L^2(\Sigma_j)$, and we conclude that
there exists $L_*>0$ and $C>0$ such that
\[
 \big\|(-\Delta - E_{\alpha}) (\chi_ju_k)\big\|_{L^2(\Sigma_j)}\le C\, e^{-\sigma L}
\]
for $L>L_*\,$.

Now let us pick any $\sigma_0\in (0,\sigma)$ and split the set $\{L:L>L_*\}$
into two disjoint parts $I_1$ and $I_2$ as follows. We say that $L\in I_1$ if
$\|\chi_j u_k\|_{L^2(\Omega)}\equiv\|\chi_j u_k\|_{L^2(\Sigma_j)}\le e^{-\sigma_0 L}\,$.
Therefore, for $L\in I_2$ we have $\|\chi_j u_k\|_{L^2(\Sigma_j)}\ge e^{-\sigma_0 L}\,$.
We check again that $\chi_j u_k\in D(M_j)\,$, so by applying
Proposition~\ref{prop8} to the operator $M_j$ we conclude that
\[
d\big(\linspan(\chi_j u_k), \ker(M_j-E_\alpha)\big)\le C_0 \, e^{-(\sigma-\sigma_0)L}\,, \quad C_0>0\,,
\]
which means that one can find $a_{jk}\in\RR$ such that
\[
\|\chi_j u_k-a_{jk}U_j\|_{L^2(\Sigma_j)}\le C_0 \, e^{-(\sigma-\sigma_0)L}\,,
\]
and
\[
|a_{jk}|\le 1+  C_0\,  e^{-(\sigma-\sigma_0)L}\,.
\]
 On the other hand, one can find $\sigma_1>0$ such that
\[
\|U_j-\psi_j\|_{L^2(\Omega)}\equiv\|U_j-\psi_j\|_{L^2(\Sigma_j)}=
\|(1-v_j)U_j\|_{L^2(\Sigma_j)}
\le C_1\,  e^{-\sigma_1L}\,.
\]
Therefore, writing $\sigma_2:=\min(\sigma_1,\, \sigma-\sigma_0)\,$, we have
\[
\|\chi_j u_k-a_j\psi_j\|_{L^2(\Omega)}=\|\chi_j u_k-a_{jk}\psi_j\|_{L^2(\Sigma_j)}\le C_2 \, e^{-\sigma_2L}
\text{ for all } L\in I_2\,.
\]
By choosing $\sigma_*:=\min(\sigma_0,\sigma_2)\,$, we conclude that, for any sufficiently large $L$,
we can find $a_j\in\RR$ with $|a_j|\le 1+ O(e^{-\sigma_*L})\,$, such that
\[
\|\chi_j u_k-a_{jk}\psi_j\|_{L^2(\Omega)}=O(e^{-\sigma_*L})\,.
\]
For $L\in I_1$ we can simply take $a_{jk}=0\,$.
We have then
\[
u_k=\sum_{j=0}^2 \chi_j u_k = \sum_{j=1}^2 a_{jk}\psi_j+O(e^{-\sigma_*L}) \text{ in } L^2(\Omega)\,.
\]
As the functions $u_k$, $k=1,2$,
form an orthonormal basis in $F$, we have $d(F,E)=O(e^{-\sigma_*L})<1$ for large $L$.
\end{proof}

\subsection{Coupling between corners}

Recall that $P_E$ denotes the orthogonal projection on $E$
in $L^2(\Omega)$. In addition, we denote by $\Pi_E$  the projection on $E$ in $L^2(\Omega)$
along $F^\perp$. The following lemma essentially reproduces Lemma~2.8 in \cite{hs1}.  We give the proof
for the sake of completeness.

\begin{lemma}\label{lem16}
For sufficiently large $L$ we have 
\[
\|\Pi_E-P_E\|=O(\sqrt{L} e^{-\beta \tau  L \cot\alpha})\,.
\]
Furthermore, we have the following identities:
\begin{itemize}
\item[(a)] $\Pi_E=\Pi_E P_F$\,,
\item[(b)] the inverse of $K:=(\Pi_E:F\to E)$ is $K^{-1}:=(P_F:E\to F)\,$,
\item[(c)] $(H_L:F\to F)=K^{-1} (\Pi_E H_L:E\to E) K\,$.
\end{itemize}
\end{lemma}

\begin{proof}
By Lemma~\ref{lem15} we can write $F=\{x+Ax: x\in E\}$, where $A$ is a bounded linear operator
acting from $E$ to $E^\perp$ with $\|A\|=O(\sqrt{L} e^{-\beta c L \cot\alpha})$. Then $F^\perp=\{y-A^*y:y\in E^\perp\}$. Furthermore, if $z=x+y$ with $x\in E$ and $y\in E^\perp$, then $P_E z=x$ and $\Pi_E z=\Tilde x$, where
$\Tilde x$ is the vector from $E$ satisfying $\Tilde x -(x+y)\in F^\perp$, which can be rewritten
as $\Tilde x-(x+y)=A^* \Tilde y- \Tilde y$ for some $\Tilde y\in E^\perp$. Considering separately the terms
in $E$ and $E^\perp$ we arrive at the system
$\Tilde x-x=A^*\Tilde y$, $y=\Tilde y$,
which implies 
\[
\big\|(P_E-\Pi_E)z \big\|=\|x-\Tilde x\|\le \|A\|\cdot \|y\|\le \|A\| \cdot\|z\|
\]
and proves the norm estimate. 

Let us check the identities. To prove (a) we write
$\Pi_E=\Pi_E (P_F+P_{F^\perp})$ and note that $\Pi_E P_{F^\perp}=0$.
To prove (b), we observe  first that the existence of the inverses
follows from Proposition~\ref{prop8}.
Now let us take any $z\in F$. It is uniquely represented
as $z=x+y$ with $x\in E$ and $y\in F^\perp$, and $P_E z=x$.
On the other hand, one has $\Pi_F x=z$, which proves the identity (b).

Furthermore, $\Pi_E H_L=\Pi_E H_L(P_F + P_{F^\perp})=\Pi_E H_L P_F +\Pi_E P_{F^\perp}H_L\,$.
Using again $\Pi_E P_{F^\perp}=0\,$,  we conclude that $\Pi_E H_L u= \Pi_E H_L P_F u$ for any $u \in E\,$.
Finally, as $H_L P_F u \in F$ for any $u\in E$, we have
\[
(\Pi_E H_L:E\to E)= (\Pi_E:F\to E) (H_L:F\to F) (P_F:E\to F)\,.
\]
Combining with (b) leads to  (c).
\end{proof}

\begin{lemma}\label{lem17}\label{prop-m}
The matrix $M$ of $\Pi_E H_L:E\to E$ in the basis $(\psi_1,\psi_2)$ is
\[
M=\begin{pmatrix}
E_\alpha & w_{12}\\
w_{21} & E_\alpha
\end{pmatrix}
+O(L^{3/2}e^{-2\beta \tau L\cot\alpha})\,, \quad L\to+\infty,
\]
where we denote
\[
w_{jk}:=\iint_\Omega v_k (U_j \nabla U_k - U_k \nabla U_j)\nabla v_j\, dx\,.
\]
\end{lemma}

\begin{proof}
The proof follows the scheme of Theorem~3.9 in~\cite{hs1}.
We have
\[
P_E u=\sum_{j,k=1}^2 c_{jk} \langle \psi_k, u\rangle \psi_j\,,
\]
where $c_{jk}$ are the coefficients satisfying
\[
\sum_{j,k=1}^2 c_{jk} \langle \psi_k, \psi_\ell \rangle \psi_j=\psi_\ell\,, \quad \ell =1,2\,,
\text{ i.e. }
\sum_{k=1}^2 c_{jk} \langle \psi_k, \psi_\ell \rangle=\delta_{jl}\,, \quad \ell =1,2\,.
\]
In other words, $(c_{jk})=G^{-1}$, where $G$ is the Gramian matrix of $(\psi_j)$, and in virtue
of Lemma~\ref{gram} we have 
\[
c_{jk}=\delta_{jk}+O(L e^{-2\beta L\cot\alpha})\,.
\]
 Therefore,
if we introduce another operator $\widehat \Pi$ by
$\widehat \Pi  u=\sum_{j=1}^2 \langle \psi_j,u\rangle \psi_j\,$,
we have 
\[
\|P_E-\widehat \Pi \|=O(L e^{-2\beta L\cot\alpha})\,.
\]
Combining with Lemma~\ref{lem16}
we obtain 
\[
\|\Pi_E-\widehat \Pi\|=O(L e^{-\beta\tau L\cot\alpha})\,.
\]
Here we used the inequality  $\tau \le 2\,$, see~\eqref{deftau}.

Now, using the structure of $\psi_j=v_j U_j$ we have
\[
H_L \psi_j = E_\alpha \psi_j -2\nabla v_j \nabla U_j -(\Delta v_j)U_j\,.
\]
The $L^2(\Omega)$-norms of two last terms on the right hand side are $O(\sqrt Le^{-\beta\tau L\cot\alpha})\,$,
which gives
\begin{equation}
       \label{eq-eap}
\begin{aligned}
\Pi_E H_L \psi_j&= \Pi_E\big(E_\alpha \psi_j \big) + \widehat \Pi \big(-2\nabla v_j \nabla U_j -(\Delta v_j)U_j\big)\\
&\quad+(\Pi_E-\widehat \Pi) \big(-2\nabla v_j \nabla U_j -(\Delta v_j)U_j\big)\\
&= E_\alpha \psi_j+ \widehat \Pi \big(-2\nabla v_j \nabla U_j -(\Delta v_j)U_j\big) +O(L^{3/2}e^{-2\beta \tau L\cot\alpha})\\
&= E_\alpha\psi_j + \sum_{k=1}^2  b_{jk} \psi_k +O(L^{3/2}e^{-2\beta \tau L\cot\alpha}).
\end{aligned}
\end{equation}
with
\[
b_{jk}:= -\iint_{\Omega} \big( 2\nabla v_j \nabla U_j +(\Delta v_j)U_j\big) \psi_k\, dx = 
-\iint_{\Omega} \big( 2\nabla v_j \nabla U_j +(\Delta v_j)U_j\big) v_k U_k\, dx.
\]
Using the Green-Riemann formula \eqref{eq-gr} we have
\begin{multline*}
\iint_{\Omega} (-\Delta v_j)U_j\big) v_k U_k \,dx=\iint_{\Omega}
\nabla v_j \nabla (U_j U_k v_k) \, dx -\int_{\partial \Omega} \dfrac{\partial v_j}{\partial n} U_j U_k v_k\, ds\\
=\iint_{\Omega} U_j U_k \nabla v_j \nabla v_k \,dx
+\iint_{\Omega} U_j v_k \nabla v_j \nabla U_k \,dx
+\iint_{\Omega} v_k U_k \nabla v_j \nabla U_j \,dx\,,
\end{multline*}
which gives
\begin{equation}
     \label{eq-bjk}
b_{jk}= \delta_{jk}w_{jk}+\varepsilon_{jk}, \quad \varepsilon_{jk}:=\iint_{\Omega} U_j U_k \nabla v_j \nabla v_k \, dx\,.
\end{equation}
Note that
\begin{equation}
       \label{eq-U1U2}
U_1(x_1,x_2)U_2(x_1,x_2)=\dfrac{2\beta^2\cos\alpha}{\sin^3\alpha}
\exp\big( -2\beta L\cot\alpha) \exp( - 2\beta x_2\big)
\end{equation}
and that $\nabla v_1 \nabla v_2$ is supported in a parallelogram of size $O(1)$
in which the value of $x_2$ is at least 
\[
S:=(\tau -1)L\cot\alpha -2/\sin\alpha\,,
\]
see Figure~\ref{fig5}. Therefore, 
\[
\varepsilon_{12}=\varepsilon_{21}=O(e^{-2 \tau \beta L\cot\alpha})\,.
\]
On the other hand, by Lemma~\ref{lem-uu} we have
\[
\varepsilon_{11}=\varepsilon_{22}=O(Le^{-2\beta \tau  L\cot\alpha})\,.
\]
Substituting these estimates into~\eqref{eq-bjk} and then into \eqref{eq-eap} leads to the conclusion.
\end{proof}

\begin{lemma}\label{lem-w}
There holds
\[
w:=w_{12}=w_{21}=\dfrac{2\beta^2 \cos^2\alpha}{\sin^4\alpha} e^{-2\beta L \cot\alpha}
+O\big(Le^{-2\beta \tau L\cot\alpha}\big)\,.
\]
\end{lemma}

\begin{figure}
\centering
\myimage{60mm}{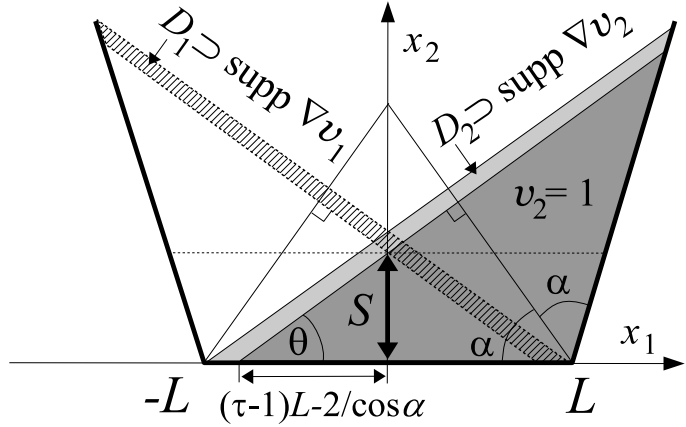}
\caption{Computation of~$S$. In the dark shaded domain there holds $v_2=1\,$, cf.~Figure~\ref{fig3}.
We have $\theta=\frac\pi 2-\alpha$ and, hence,
$S=\big((\tau -1)L-2\cos\alpha\big)\tan\theta\equiv (\tau -1)L\cot\alpha-2/\sin\alpha\,.$
\label{fig5}}
\end{figure}

\begin{proof}
The equality $w_{12}=w_{21}$ follows from the symmetry considerations.
Furthermore, we have the equality
\[
U_1\nabla U_2 - U_2\nabla U_1= 2\beta \cot \alpha \begin{pmatrix} 1\\ 0 \end{pmatrix} U_1 U_2\,.
\]
Substituting the expression for $U_1U_2$ from~\eqref{eq-U1U2} we obtain
\[
w_{12}=
\dfrac{4\beta^3 \cos^2\alpha}{\sin^4\alpha} e^{-2\beta L \cot\alpha} A\,, \quad
A:=
\iint_\Omega e^{-2\beta x_2} v_2  \dfrac{\partial v_1}{\partial x_1} \,dx\,.
\]
Using the explicit construction of $v_1$ and $v_2$ we can see that, for \break 
$x_2<S:=(\tau -1)L\cot\alpha -2/\sin\alpha\,$, we have the following property: if $(x_1,x_2)\in \supp\nabla v_1$, then $v_2(x_1,x_2)=1\,$, see Figure~\ref{fig5}.
This allows one to estimate $A$ by
\[
A=\iint_{\Omega\cap\{x_2\le S\}} e^{-2\beta x_2} \dfrac{\partial v_1 (x_1,x_2)}{\partial x_1} \,dx
+O\big(Le^{-2\beta(\tau -1)L\cot\alpha}\big)\,.
\]
On the other hand, by Fubini
\[
\iint_{\Omega\cap\{x_2\le S\}} e^{-2\beta x_2} \dfrac{\partial v_1 (x_1,x_2)}{\partial x_1} dx=
\int_0^S e^{-2\beta x_2} \bigg(\int \dfrac{\partial v_1(x_1,x_2)}{\partial x_1}dx_1\bigg) dx_2.
\]
The interior integral is equal to $1$ for any $x_2$, which finally gives
\[
A=\int_0^S e^{-2\beta x_2} dx_2+O\big(Le^{-2\beta(\tau -1)L\cot\alpha}\big)
=\dfrac{1}{2\beta}+O\big(Le^{-2\beta(\tau -1)L\cot\alpha}\big).
\]
\end{proof}

\begin{lemma}\label{lem-n}
The matrix $N$ of $\Pi_E H_L:E\to E$ in  the orthonormal basis
\[
\phi_k=\sum_{j=1}^2 \psi_j \sigma_{jk}, \quad k=1,2, \quad \sigma:=(\sigma_{jk}):=\sqrt{G^{-1}}\,,
\]
has the form
\[
N=N_0+O(L^2e^{-2\beta \tau  L\cot\alpha}) \quad \text{with} \quad
N_0=\begin{pmatrix}
E_\alpha & w\\
w & E_\alpha
\end{pmatrix}\,.
\]
Here $G$ is the Gramian matrix from Lemma~\ref{gram}.
\end{lemma}

\begin{proof}
Due to Lemma~\ref{gram} we have
$G= I +T$ with $T=O(Le^{-2\beta L\cot\alpha})$,
which shows that
\[
\sigma= I -\frac12 T+O(L^2e^{-4\beta L\cot\alpha})\,, \quad
\sigma^{-1}= I +\frac12 T+O(L^2e^{-4\beta L\cot\alpha})\,.
\]
On the other hand, using the matrix $M$ from Lemma~\ref{prop-m}, we have
$N=\sigma^{-1}M \sigma$. So we get
\begin{align*}
N&=\Big(I +\frac12 T +O(L^2e^{-4\beta L\cot\alpha}\Big)
\Big( E_\alpha + \begin{pmatrix} 0 &w\\
w & 0\end{pmatrix}
+O(L^{3/2}e^{-2\beta t L \cot \alpha})
\Big)\\
&\quad \times\Big(I -\frac12 T +O(L^2e^{-4\beta L\cot\alpha}\Big)\\
&= \begin{pmatrix}
E_\alpha & w\\
w & E_\alpha
\end{pmatrix}
+\dfrac 12 \bigg[
T
 \begin{pmatrix} 0 & w\\
w & 0
\end{pmatrix}
-
 \begin{pmatrix} 0 & w\\
w & 0
\end{pmatrix}
T
\bigg]
+O(L^2e^{-2\beta \tau  L\cot\alpha})\,.
\end{align*}
The term in the square brackets equals zero due to Lemma~\ref{gram},
and this achieves the proof.
\end{proof}

\begin{proof}[\bf Proof of Theorem~\ref{thm1}]
Now we are able to finish the proof  of the main theorem.
The eigenvalues of the matrix $N_0$ from Lemma~\ref{lem-n}
are $E_\pm:=E_\alpha\pm |w|$, and in view of Lemma~\ref{lem-w}
we have
\[
E_\pm=-\dfrac{\beta^2}{\sin^2\alpha}
\pm \dfrac{2\beta^2 \cos^2\alpha}{\sin^4\alpha} e^{-2\beta L \cot\alpha}
+O\big(Le^{-2\beta \tau L\cot\alpha}\big)\,.
\]
By Lemma~\ref{lem-w}, these numbers $E_\pm$
coincide up to $O(L^2e^{-2\beta \tau L\cot\alpha})$
with the eigenvalues of $H_L$ in $I$, which are exactly
$E_1(L)$ and $E_2(L)\,$. It remains
to apply elementary trigonometric identities to pass from $\alpha=\omega/2$
to $\omega\,$.
\end{proof}

\section{Conclusion}\label{rmk}
To conclude this article, let us add a few remarks.

\begin{remark}\label{rem41}
The family of  operators $H_L$ includes one case in which one can separate the variables,
namely, the case $\omega=\frac\pi2\,$, for which the estimate of Theorem~\ref{thm1} takes the form
\begin{equation}
         \label{1d-our}
E_{1/2}(L)=-2\beta^2 \mp 4\beta^2 e^{-2\beta L}+O(L^2e^{-4\beta L})\,.
\end{equation}
On the other hand, one can represent
$H_L=A\otimes 1 + 1\otimes B_L$, where $A$ and $B_L$ are operators
in $L^2(0,\infty)$ and $L^2(-L,L)$ respectively:
\begin{gather*}
A u=-u'', \quad D(A)=\big\{
u\in H^2(0,\infty): u'(0)+\beta u(0)=0
\big\}\,,\\
B_Lv=-v''\,, \quad
D(B_L)=\big\{
v\in H^2(-L,L): v'(-L)+\beta v(-L)=v'(L)-\beta v(L)=0
\big\}\,.
\end{gather*}
One easily computes 
\[
\spec A=\{-\beta^2\}\cup[0,+\infty)\,.
\]
 On the other hand,
$B_L$ has a compact resolvent and, if one denotes its eigenvalues by $\varepsilon_j(L)$,
then 
\[
E_j(L)=-\beta^2+\varepsilon_j(L)\,.
\]
The behavior of $\varepsilon_j(L)$, $j=1,2$, can be studied in a rather explicit way
by using the 1D nature of the problem, see Proposition~\ref{prop-robin1d} in the appendix,
and one gets
\[
E_{1/2}(L)= -2\beta^2\mp4\beta^2e^{-2\beta L}+8\beta^2(2\beta L-1)e^{-4\beta L}+O(L^2 e^{-6\beta L})\,,\\
\]
One observes that the remainder estimate in our asymptotics \eqref{1d-our} only
differs by the factor $L$ from the exact one.
\end{remark}

\begin{remark}
One can also consider the case $\omega=\frac\pi 3\,$, i.e. the case of the equilateral triangle.
In this case one has an interaction between the three corners. The above scheme
works in essentially the same way; see also~\cite{hs2} and \cite[Section 16.2]{hf} for the general discussion.
One can prove that, 
for sufficiently large $L$, there exists a bijection $\sigma$ between the three lowest eigenvalues
of $H_L$ and the three eigenvalues of the matrix 
\[
N_0=\begin{pmatrix}
E_\alpha & w & w\\
w & E_\alpha & w\\
w & w & E_\alpha
\end{pmatrix},
\quad
w=24\, \beta^2e^{-2\sqrt 3 L}\,,
\]
such that $\sigma (E)=E+O(L^2 e^{-4\sqrt 3\beta L})\,$.

Note that the eigenvalues of $N_0$ are $E_\alpha-w$ (simple)
and $E_\alpha+w$ (double), which means that the three
lowest eigenvalues of $H_L$ behave as
\begin{align*}
E_1(L)&=-4\, \beta^2-24\,\beta^2e^{-2\sqrt 3 L}+O(L^2 e^{-4\sqrt 3\beta L})\,,\\
E_j(L)&=-4\,\beta^2+24\, \beta^2e^{-2\sqrt 3 L}+O(L^2 e^{-4\sqrt 3\beta L})\,, \quad j=2,3\,,
\end{align*}
i.e. no splitting is visible between $E_2$ and $E_3\,$. Actually there is no surprise,
as a symmetry argument as well as the explicit formulas from~\cite[Section 7]{mc} show that
\[
E_2(L)=E_3(L).
\]
\end{remark}

\begin{remark}
One may see from the proof that the result admits direct extensions to a little bit more general domains.
Namely, assume that $\Omega=L\Omega'$ with some $L$-independent $\Omega'$ and such that
$\Omega$ coincides with $\Omega_L$ near the axis $Ox_1$ in the following sense:
one still can construct the triangles $\Theta_j(\tau L)$, $j=1,2$, as in~Subsection~\ref{ss-cutoff}
for some $\tau >1$, and $\Omega$ does not contain any further corner whose opening
is smaller or equal to $\omega$. Then Theorem~\ref{thm1} is valid
for the first two eigenvalues of $H(\Omega,\beta)$ with $\delta=2(\tau -1)$.
It would be interesting to know if any result of this kind can be obtained for more general domains
and more general relative positions of the corners. For the smooth domains, one may expect that
the role of the corners is played by the points of the boundary at which the curvature is maximal~\cite{exner,kp},
which gives rise to similar questions. This is actually the case for surface superconductivity, see \cite{hf} and references therein.
\end{remark}

\begin{remark}
Our considerations were in part stimulated by the paper~\cite{BND} which studies the asymptotic
behavior of the eigenvalues of the magnetic Neumann Laplacians in curvilinear polygons, but in our case
we were able to obtain a more precise result due to the fact that we know 
the exact eigenfunction of an infinite sector. One may wonder if 
our machinery can help to progress in the problem of~\cite{BND}.
We note that both the magnetic Neumann Laplacian
and the Robin Laplacian appear as approximate models in the theory of surface superconductivity
and are closely related to the computation of the critical temperature~\cite{gs1,hs1}.
\end{remark}

\appendix

\section{1D Robin problem}\label{1d}

In this section, we study the one-dimensional Robin problem. The expressions obtained have their own interest,
but some estimates can be used to obtain a better estimate for the analysis of the two-dimensional situation, as explained in Remark~\ref{rem41}.

\begin{lemma}\label{lem1dn}
For $\beta>0$ and $\ell>0$, denote by $N_{\beta,\ell}$ the operator
acting in $L^2(0,\ell)$ as $f\mapsto -f''$ on the functions
$f\in H^2(0,\ell)$ satisfying the boundary conditions
$f'(0)=0$ and $f'(\ell)=\beta f(\ell)$. Then the lowest eigenvalue
$E_N(\beta,\ell)$ is the unique strictly negative eigenvalue, and
\begin{equation}
   \label{eq-1dn}
E_N(\beta,\ell)=-\beta^2-4\beta^2e^{-2\beta \ell}+8\beta^2(2\beta\ell-1)e^{-4\beta\ell}+O(\ell^2 e^{-6\beta\ell})\,
\text{ as $\ell$ tends to $+\infty$\,,}
\end{equation}
and the associated eigenfunction is $x\mapsto \cosh(\sqrt{-E_N(\beta,\ell)} x)\,$.
\end{lemma}

\begin{proof}

Let us write $E_N(\beta,\ell)=-k^2$ with $k>0$. The associated eigenfunction $f$
must be of the form $f(x)=Ae^{kx}+Be^{-kx}$
with some $(A,B)\in\RR^2\setminus\big\{(0,0)\big\}$. Taking into the account 
the boundary conditions we get the linear system
\[
A-B=0\,,\, (k-\beta)e^{k\ell}A-(k+\beta)e^{-k\ell}B=0\,.
\]
It follows that $f(x)=2B \cosh(kx)$. The system has non-trivial solutions iff 
\begin{equation}
   \label{eq-n1}
   (k-\beta)e^{k\ell}=(k+\beta)e^{-k\ell}\,.
\end{equation}
This can be rewritten as $k\ell \tanh (k\ell)=\beta\ell$.
One easily checks that the function
\[
(0,+\infty)\ni t\mapsto t \tanh t\in(0,+\infty)
\]
is a bijection, which means that the solution $k$ to ~\eqref{eq-n1}
is defined uniquely, which shows that we have exactly one negative eigenvalue.

To calculate its asymptotics, we first take into account the signs of all terms in~\eqref{eq-n1},
which gives $k>\beta\,$.

Rewriting \eqref{eq-n1} in the form
\[
(k-\beta)  = 2 \beta e^{-2k\ell} / (1- e^{-2k\ell}) =   2 \beta e^{-2 \beta \ell} e^{-2(k-\beta) \ell} / (1- e^{-2(k-\beta)\ell}e^{-2 \beta \ell}) \,,
\]
we get that 
\begin{equation}\label{eq-n2}
k-\beta = O (e^{-2 \beta \ell})\,.
 \end{equation}
It follows also from \eqref{eq-n1} that
\begin{equation}
        \label{eq-n3}
k=\dfrac{1+e^{-2k\ell}}{1-e^{-2k\ell}}\beta=
\big(1+2e^{-2k\ell}+O(e^{-4k\ell})\big)\beta\,,
\quad \ell\to +\infty\,.
\end{equation}
Implementing \eqref{eq-n2}, we infer that
\begin{equation}
       \label{eq-kaux1}
k=\big(1+2e^{-2\beta \ell}+O(\ell e^{-4\beta\ell})\big)\beta
=\beta +2\beta e^{-2\beta \ell}+O(\ell e^{-4\beta\ell})\,.
\end{equation}
By taking an additional term in \eqref{eq-n3},
\[
k=\dfrac{1+e^{-2k\ell}}{1-e^{-2k\ell}}\beta=
\big(1+2e^{-2k\ell}+2e^{-4k\ell}+O(e^{-6k\ell})\big)\beta\,,
\quad \ell\to +\infty\,,
\]
and by using~\eqref{eq-kaux1} one gets
\begin{equation}\label{dern}
k=\beta+2\beta e^{-2\beta\ell}+2\beta(1-4\beta\ell)e^{-4\beta\ell}+O(\ell^2e^{-6\beta\ell})\,.
\end{equation}
Computing $E=-k^2$ gives the result.
\end{proof}

\begin{lemma}\label{lem1dd}
For $\beta>0$ and $\ell>0$, denote by $D_{\beta,\ell}$ the operator
acting in $L^2(0,\ell)$ as $f\mapsto -f''$ on the functions
$f\in H^2(0,\ell)$ satisfying the boundary conditions
$f(0)=0$ and $f'(\ell)=\beta f(\ell)$, and let $E_D(\beta,\ell)$ denote its lowest eigenvalue.
Then $E_D(\beta,\ell)<0$ iff $\beta\ell>1$, and in that case it is
the only negative eigenvalue. Furthermore,
\begin{equation}
   \label{eq-1dd}
E_D(\beta,\ell)=-\beta^2+4\beta^2e^{-2\beta \ell}+8\beta^2(2\beta\ell-1)e^{-4\beta\ell}+O(\ell^2 e^{-6\beta\ell})
\text{ as $\ell$ tends to $+\infty\,$,}
\end{equation}
and the associated eigenfunction is $x\mapsto \sinh(\sqrt{-E_D(\beta,\ell)} x)$.
\end{lemma}

\begin{proof}
Let us write $E_D(\beta,\ell)=-k^2$ with $k>0$.
The associated eigenfunction $f$ is of the form $f=Ae^{kx}+Be^{-kx}$
with some $(A,B)\in\RR^2\setminus\big\{(0,0)\big\}$. Taking into the account 
the boundary conditions we get the linear system
\[
A+B=0\,,\, (k-\beta)e^{k\ell}A-(k+\beta)e^{-k\ell}B=0\,,
\]
which gives the representation $f(x)=2A\sinh(kx)$.
Non-trivial solutions exist iff 
\begin{equation}
   \label{eq-d1}
(\beta+k)e^{-k\ell}=(\beta-k)e^{k\ell}\,.   
\end{equation}
The preceding equation can be rewritten as
\[
k\ell\coth(k\ell)=\beta\ell\,.
\]
One easily checks that the function
\[
(0,+\infty)\ni t \mapsto t \coth t \in(1,+\infty)
\]
is a bijection, which shows that \eqref{eq-d1} has a solution iff $\beta\ell>1$,
and if it is the case, the solution is unique, which gives in turn
the unicity of the negative eigenvalue.

For the rest of the proof we assume that 
\[
\beta\ell>1\,.
\]
By considering the signs of both sides in \eqref{eq-d1} we conclude that $k<\beta$. 
Furthermore, we may rewrite \eqref{eq-d1} as $g(k)=0$
with
\[
g(k)=\log(\beta+k)-\log(\beta-k)-2k\ell\,.
\]
We have $g(0+)=0$ and $g(\beta-)=+\infty\,$.
The equation $g'(k)=0$ takes the form
\[
\beta^2-k^2=\dfrac{\beta}{\ell},
\]
and its unique solution is
\[
k^*= \beta\sqrt{1-\dfrac{1}{\beta\ell}}\,.
\]
It follows that the equation $g(k)=0$ has a unique solution $k$
in $(0,\beta)$ and that $k\in(k^*,\beta)\,$. 
On the other hand,  we obtain the  estimate
\[
k^*>\beta\Big(1-\dfrac{1}{\beta\ell}\Big)=\beta-\dfrac{1}{\ell}.
\]
Hence, the solution of $g(k)=0$ satisfies
\begin{equation}\label{rough}
\beta - \dfrac{1}{\ell} < k < \beta\,.
\end{equation}
We rewrite \eqref{eq-d1} in the form
 \[
 \beta -k = \dfrac{2 k}{e^{2 k \ell}- 1}\,.
 \]
 and we deduce with the help of~\eqref{rough} that
 \[
 \beta -k = O (e^{-2 \beta \ell}) \text{ as } \ell \rightarrow +\infty.
 \]
By going through the same steps
as in the proof of Lemma~\ref{lem1dn}, one gets the result.
\end{proof}

\begin{prop}\label{prop-robin1d}
For $\beta>0$ and $\ell>0$, let $B_{\ell}$ denote
the operator
acting in $L^2(-\ell,\ell)$ as $f\mapsto -f''$ on the functions
$f\in H^2(-\ell,\ell)$ satisfying the boundary conditions \break 
$f'(\pm \ell)=\pm \beta f(\pm \ell)\,$, and
let $E_1(\ell)$ and $E_2(\ell)$ be the two lowest eigenvalues,
$E_1(\ell)<E_2(\ell)\,$. Then:
\begin{itemize}
\item $E_1(\ell)<0\,$,
\item $E_2(\ell)<0$ iff $\beta\ell>1\,$,
\item all other eigenvalues are non-negative.
\end{itemize}
Furthermore,
\begin{align*}
 E_1(\ell)&=-\beta^2-4\beta^2e^{-2\beta \ell}+8\beta^2(2\beta\ell-1)e^{-4\beta\ell}+O(\ell^2 e^{-6\beta\ell})\,,\\
E_2(\ell)&= -\beta^2+4\beta^2e^{-2\beta \ell}+8\beta^2(2\beta\ell-1)e^{-4\beta\ell}+O(\ell^2 e^{-6\beta\ell})\,,
\end{align*}
as $\ell$ tends to $+\infty$. The respective eigenfunctions $f_1$ and $f_2$
are
\[
f_1(x)=\cosh \big(\sqrt{-E_1(\ell)}x\big)\,,
\quad
f_2(x)=\sinh \big(\sqrt{-E_2(\ell)}x\big)\,.
\]

\end{prop}

\begin{proof}
Let us use the notation of Lemmas~\ref{lem1dn} and~\ref{lem1dd}.
Note that:
\begin{itemize}
\item $B_{\ell}$ commutes with the reflections with respect to the origin,
\item its first eigenfunction $f_1$ is non-vanishing and even, hence, $f_1'(0)=0\,$,
\item its second eigenfunction $f_2$ has one zero in $(-\ell,\ell)$
and is odd, hence $f_2(0)=0\,$.
\end{itemize}
Therefore, $E_1(\ell)=E_N(\beta,\ell)$ and $E_2(\ell)=E_D(\beta,\ell)\,$, and the result follows
from Lemmas~\ref{lem1dn} and~\ref{lem1dd}.
\end{proof}

\end{document}